\documentclass{amsart}
\usepackage{amssymb,amsmath,amsthm}
\usepackage[dvips]{graphicx}

\usepackage{graphicx,leftidx,undertilde}      

\usepackage{amsmath}

\usepackage{amssymb}
\let\oldmarginpar\marginpar
\renewcommand\marginpar[1]{\-\oldmarginpar[\raggedleft\footnotesize #1]%
{\raggedright\footnotesize #1}}

\begin{document}

\newtheorem{theorem}{Theorem}[section]
\newtheorem{corollary}[theorem]{Corollary}
\newtheorem{lemma}[theorem]{Lemma}
\newtheorem{proposition}[theorem]{Proposition}
\theoremstyle{definition}
\newtheorem{definition}[theorem]{Definition}
\theoremstyle{remark}
\newtheorem{remark}[theorem]{Remark}
\theoremstyle{definition}
\newtheorem{example}[theorem]{Example}
\theoremstyle{acknowledgment}
\newtheorem{acknowledgment}[theorem]{Acknowledgment}

\numberwithin{equation}{section}

\title[On CR-statistical submanifolds of holomorphic]{On CR-statistical submanifolds of \\holomorphic statistical manifolds}

\author[A.N. Siddiqui]{Aliya Naaz Siddiqui}
\address{Department of Mathematics, \\ Faculty of Natural Sciences, \\ Jamia Millia Islamia, \\ New Delhi - 110025, India}
\email{aliyanaazsiddiqui9@gmail.com}

\author[F.R. Al-Solamy]{Falleh R. Al-Solamy}
\address{Department of Mathematics, \\ Faculty of Science, \\ King Abdulaziz University, \\ Jeddah 21589, \\ Saudi Arabia}
\email{falleh@hotmail.com}

\author[M.H. Shahid]{Mohammad Hasan Shahid}
\address{Department of Mathematics, \\ Faculty of Natural Sciences, \\ Jamia Millia Islamia, \\ New Delhi - 110025, India}
\email{hasan$\_$jmi@yahoo.com}

\author[I. Mihai]{Ion Mihai}
\address{Department of Mathematics, \\ Faculty of Mathematics and Computer Science,\\ University of Bucharest, \\ Bucharest, 010014, Romania}
\email{imihai@fmi.unibuc.ro}



\begin{abstract}
In the present paper, we investigate some properties of the distributions involved in the definition of a CR-statistical submanifold. The characterization of a CR-product in holomorphic statistical manifolds is given. By using an optimization technique, we establish a relationship between the Ricci curvature and the squared norm of the mean curvature of any submanifold in the same ambient space. The equality case is also discussed here. This paper finishes with some related examples.
\end{abstract}

\maketitle

\section{Introduction}

In 1985, Amari \cite{article.1} introduced the notion of statistical manifolds in the context of information geometry. They may be considered as manifolds consisting of certain probability density functions. Geometrically, they are Riemannian manifolds with a certain affine connection. Beyond the expectations, statistical manifolds are familiar to researchers because they have appeared with alternate names in various research topics in affine geometry \cite{article.09} and in Hessian geometry \cite{article.8}. These manifolds have applications in document classification, face recognition, image analysis, clustering and so on. The oldest examples of statistical structures are the induced structures (consisting of the second fundamental form and the induced connection) on locally strongly convex hypersurfaces in $\mathbb{R}^{n+1}$ endowed with an equiaffine transversal vector field (in other words - with relative normalization). It is natural for researchers to try to build the submanifold theory and the complex manifold theory of statistical manifolds.

In 1978, A. Bejancu \cite{article.2} introduced the notion of a CR-submanifold of a Kaehler manifold with complex structure $\mathcal{J}$. The investigation of CR-submanifolds of Kaehler manifolds has a long history. Many geometers investigated CR-submanifolds in different ambient manifolds such as Hermitian manifolds, Kaehler manifolds, Sasakian manifolds, complex space forms, complex projective space, etc. In this direction, H. Furuhata and I. Hasegawa \cite{article.7} intensively studied their statistical version, that is, CR-statistical submanifolds in holomorphic statistical manifolds. Recently, M.N. Boyom et al. \cite{article.700} studied classification of totally umbilical CR-statistical submanifolds in holomorphic statistical manifolds with constant holomorphic curvature. Also, Milijevi$\acute{c}$ \cite{4mir} and Aliya et al. \cite{article} studied totally real statistical submanifolds in holomorphic statistical manifolds, independently.

The Curvatures invariants are widely used in the field of differential geometry and in physics also. The Ricci curvature is the essential term in the Einstein field equations, which plays a key role in general relativity. It is immensely studied in differential geometry as it gives a way of measuring the degree to which the geometry determined by a given Riemannian metric might differ from that of ordinary Euclidean $n$-space. A Riemannian manifold is said to be an Einstein manifold if the Ricci tensor satisfies the vacuum Einstein equation. The lower bounds on the Ricci tensor on a Riemannian manifold enable one to find global geometric and topological information by comparison with the geometry of a constant curvature space form. Recently, Aydin et al. \cite{a} investigated Chen-Ricci inequality for submanifolds with any codimension of statistical manifolds of constant curvature. Also, A. Mihai et al. established similar inequality with respect to a sectional curvature of the ambient Hessian manifold in \cite{a2}.

In this work, we wish to add some more results to CR-statistical submanifolds in holomorphic statistical manifolds, which are new objects originating from information geometry. We see whether the results in classical Riemannian case hold in our settings or not. Also, we obtain a Chen-Ricci inequality for CR-statistical submanifolds in a holomorphic statistical manifold of constant holomorphic sectional curvature by following an optimization method (see Section 5).

\section{Some Preliminaries}

In this section, we give some basic definitions and fundamental formulae which are related to this paper.

\begin{definition}
\cite{article.7} A statistical manifold is a Riemannian manifold $(\overline{N},g)$ endowed with a pair of torsion-free affine connections $\overline{\nabla}$ and $\overline{\nabla}^{*}$ satisfying
\begin{eqnarray*}
Xg(Y, Z) = g(\overline{\nabla}_{X}Y, Z) + g(Y, \overline{\nabla}_{X}^{*}Z),
\end{eqnarray*}
for any $X, Y, Z \in \Gamma(T\overline{N})$ and $\overline{\nabla}g$ is symmetric. The connections $\overline{\nabla}$ and $\overline{\nabla}^{*}$ are called dual connections and satisfy $\big(\overline{\nabla}^{*}\big)^{*} = \overline{\nabla}$.
\end{definition}

\begin{remark}
\cite{article.7}
\begin{enumerate}
\item[(a)] If $(\overline{\nabla}, g)$ is a statistical structure on $\overline{N}$, then $(\overline{\nabla}^{*}, g)$ is also a statistical structure.
\item[(b)] Any torsion-free affine connection $\overline{\nabla}$ always has a dual connection given by
\begin{eqnarray*}
2\overline{\nabla}^{0} = \overline{\nabla} + \overline{\nabla}^{*},
\end{eqnarray*}
where $\overline{\nabla}^{0}$ is the Levi-Civita connection on $\overline{N}$.
\end{enumerate}
\end{remark}

\begin{definition}
\cite{article.7} Let $\overline{N}$ be a Kaehler manifold with almost complex structure $\mathcal{J} \in \Gamma(T\overline{N}^{(1,1)})$ and metric $g$. A quadruple $(\overline{N}, \overline{\nabla}, g, \mathcal{J})$ is called a holomorphic statistical manifold if
\begin{enumerate}
\item[(a)] $(\overline{\nabla}, g)$ is a statistical structure on $\overline{N}$,
\item[(b)] a $2$-form $\omega$ on $\overline{N}$, given by
\begin{eqnarray*}
\omega(X, Y) = g(X, \mathcal{J}Y),
\end{eqnarray*}
for any $X, Y \in \Gamma(T\overline{N})$, is $\overline{\nabla}-$parallel, that is, $\overline{\nabla}\omega = 0$.
\end{enumerate}
\end{definition}

\begin{remark}
\cite{article.7}
\begin{enumerate}
\item[(a)] For a holomorphic statistical manifold $\big(\overline{N}, \mathcal{J}, g\big)$, we have
\begin{eqnarray}\label{c}
\overline{\nabla}_{X}(\mathcal{J}Y) = \mathcal{J}\overline{\nabla}_{X}^{*}Y,
\end{eqnarray}
for any $X, Y \in \Gamma(T\overline{N})$.
\item[(b)] Since $\omega$ is skew-symmetric, we have $g(X, \mathcal{J}Y) = - g(\mathcal{J}X, Y)$, for any $X, Y \in \Gamma(T\overline{N})$.
\item[(c)] A holomorphic statistical manifold is nothing but a special Kaehler manifold if $\overline{\nabla}$ is flat.
\end{enumerate}
\end{remark}

Let $(\overline{N}, \overline{\nabla}, g)$ be a statistical manifold and $N$ be a submanifold of $\overline{N}$. By $T_{x}^{\perp}N$ we denote
the normal space of $N$, i.e., $T_{x}^{\perp}N = \big\{v \in T_{x}\overline{N} \big | g(u, v) = 0, u \in T_{x}N\big\}$. Define \cite{article.7}
\begin{equation}\label{e}
\left .
\begin{tabular}{ccc}
$\overline{\nabla}_{X}Y = \nabla_{X}Y + \mathcal{B}(X, Y)$,\\
$\overline{\nabla}_{X}^{*}Y = \nabla_{X}^{*}Y + \mathcal{B}^{*}(X, Y)$,\\
$\overline{\nabla}_{X}V = - \mathcal{A}_{V}(X) + \nabla_{X}^{\perp}V$,\\
$\overline{\nabla}_{X}^{*}V = - \mathcal{A}_{V}^{*}(X) + \nabla_{X}^{\perp *}V$,
\end{tabular}
\right \}
\end{equation}
for any $X, Y \in \Gamma(TN)$ and $V \in \Gamma(T^{\perp}N)$. Here $\overline{\nabla}$ and $\overline{\nabla}^{*}$ (respectively, $\nabla$ and $\nabla^{*}$) are the dual connections on $\overline{N}$ (respectively, on $N$), $\mathcal{B}$ and $\mathcal{B}^{*}$ are symmetric and bilinear, called the imbedding curvature tensor of $N$ in $\overline{N}$ for $\overline{\nabla}$ and the imbedding curvature tensor of $N$ in $\overline{N}$ for $\overline{\nabla}^{*}$, respectively. Since $\mathcal{B}$ and $\mathcal{B}^{*}$ are bilinear, the linear transformations $\mathcal{A}_{V}$ and $\mathcal{A}_{V}^{*}$ are related to the imbedding curvature tensors by \cite{article.7}
\begin{eqnarray}\label{f}
\left .
\begin{tabular}{ccc}
$g(\mathcal{B}(X, Y), V) = g(\mathcal{A}_{V}^{*}(X), Y)$,\\
$g(\mathcal{B}^{*}(X, Y), V) = g(\mathcal{A}_{V}(X), Y)$,
\end{tabular}
\right \}
\end{eqnarray}
for any $X, Y \in \Gamma(TN)$ and $V \in \Gamma(T^{\perp}N)$.

Let $\overline{R}$ and $R$ be the curvature tensor fields with respect to $\overline{\nabla}$ and $\nabla$, respectively. Then the Gauss equation is given by \cite{article.7}
\begin{eqnarray}\label{g}
g(\overline{R}(X, Y) Z, W) &=& g(R(X, Y) Z, W) + g(\mathcal{B}(X, Z), \mathcal{B}^{*}(Y, W)) \nonumber \\
&&- g(\mathcal{B}^{*}(X, W), \mathcal{B}(Y, Z)),
\end{eqnarray}
for any $X, Y, Z, W \in \Gamma(TN)$. Similarly, $\overline{R}^{*}$ and $R^{*}$ are respectively the curvature tensor fields with respect to $\overline{\nabla}^{*}$ and $\nabla^{*}$. Then the curvature tensor fields of $\overline{N}$ and $N$ are respectively given by \cite{article.7}
\begin{eqnarray}\label{g1}
\overline{\mathcal{S}} = \frac{1}{2}(\overline{R} + \overline{R}^{*}),  \hspace{0.2 cm} \textrm{and} \hspace{0.2 cm}  \mathcal{S} = \frac{1}{2}(R + R^{*}).
\end{eqnarray}

Suppose that $dim(N)= m$ and $dim(\overline{N}) = 2n$. We consider a local orthonormal tangent frame $\{e_{1}, \dots, e_{m}\}$ of $T_{x}N$ and a local orthonormal normal frame $\{e_{m+1}, \dots, e_{2n}\}$ of $T_{x}^{\perp}N$ in $\overline{N}$, $x \in N$. Then the mean curvature vectors $H$ and $H^{*}$ of $N$ in $\overline{N}$ are
$$H = \frac{1}{m}\sum_{i = 1}^{m}\mathcal{B}(e_{i}, e_{i}),$$
and
$$H^{*} = \frac{1}{m}\sum_{i = 1}^{m}\mathcal{B}^{*}(e_{i}, e_{i}).$$
Also, we set
$$h_{ij}^{r} = g(\mathcal{B}(e_{i}, e_{j}), e_{r}),$$
and
$$h_{ij}^{*r} = g(\mathcal{B}^{*}(e_{i}, e_{j}), e_{r}),$$
for $i, j \in \{1, \dots, m\}$, $r \in \{m+1, \dots, 2n\}$.

Let $X$ be a unit vector such that $||X|| = 1$. We choose an orthonormal frame $\{e_{1}, \dots, e_{m}\}$ of $T_{x}N$, $x \in N$, such that $e_{1} = X$. Then the Ricci curvature at $X$ is given by
\begin{eqnarray*}
\overline{Ric}^{\overline{\nabla}, \overline{\nabla}^{*}}(X) &=& \sum_{i=2}^{m}\overline{\mathcal{K}}^{\overline{\nabla}, \overline{\nabla}^{*}}(X \wedge e_{i}) \\
&=& \frac{1}{2}\bigg\{\sum_{i=2}^{m}\overline{\mathcal{K}}(X \wedge e_{i}) + \sum_{i=2}^{m}\overline{\mathcal{K}}^{*}(X \wedge e_{i})\bigg\},
\end{eqnarray*}
where $\overline{\mathcal{K}}(e_{i} \wedge e_{j})$ denotes the sectional curvature of the 2-plane section spanned by $e_{i}$ and $e_{j}$.

For any $X \in \Gamma(TN)$, we put \cite{article.9}
\begin{eqnarray}
\mathcal{J}X = \mathcal{P}X + \mathcal{F}X,
\end{eqnarray}
where $\mathcal{P}X$ and $\mathcal{F}X$ are the tangential and normal components of $\mathcal{J}X$, respectively. Then $\mathcal{P}$ is an endomorphism of the tangent bundle $TN$ and $\mathcal{F}$ is a normal-bundle-valued $1-$form on $TN$.

In the similar way, for any $V \in \Gamma(T^{\perp}N)$, we put \cite{article.9}
\begin{eqnarray}
\mathcal{J}V = tV + fV,
\end{eqnarray}
where $tV$ and $fV$ are the tangential and normal components of $\mathcal{J}V$, respectively. Then $f$ is an endomorphism of the normal bundle $T^{\perp}N$ and $t$ is a tangent bundle-valued $1-$form on $T^{\perp}N$.

The statistical version of the definition of a CR-submanifold is as follows:

\begin{definition}
\cite{article.7} A statistical submanifold $N$ is called a CR-statistical submanifold in a holomorphic statistical manifold $\overline{N}$ of dimension $2m \geq 4$ if $N$ is CR-submanifold in $\overline{N}$, i.e., there exists a differentiable distribution $\mathcal{D} : x \rightarrow \mathcal{D}_{x} \subseteq T_{x}N$ on $N$ satisfying the following conditions:
\begin{enumerate}
\item[(a)] $\mathcal{D}$ is holomorphic, i.e., $\mathcal{J}\mathcal{D}_{x} = \mathcal{D}_{x} \subseteq T_{x}N$ for each $x \in N$, and
\item[(b)] the complementary orthogonal distribution $\mathcal{D}^{\perp} : x \rightarrow \mathcal{D}_{x}^{\perp} \subseteq T_{x}N$ is
totally real, i.e., $\mathcal{J}\mathcal{D}_{x}^{\perp} \subset T_{x}^{\perp}N$ for each $x \in N$.
\end{enumerate}
\end{definition}

\begin{remark}\label{q1}
\cite{article.7} CR-statistical submanifolds are characterized by the condition $\mathcal{FP} = 0$.
\end{remark}

\begin{definition}
\cite{article.7} Let $N$ be a CR-statistical submanifold of a holomorphic statistical manifold $\overline{N}$. Then $N$ is said to be
\begin{itemize}
\item[(a)] mixed totally geodesic with respect to $\overline{\nabla}$ if $\mathcal{B}(X, Y) = 0$ for any $X \in \mathcal{D}$ and $Y \in \mathcal{D}^{\perp}$.
\item[(a)$^*$] mixed totally geodesic with respect to $\overline{\nabla}^{*}$ if $\mathcal{B}^{*}(X, Y) = 0$ for any $X \in \mathcal{D}$ and $Y \in \mathcal{D}^{\perp}$.
\end{itemize}
\end{definition}

\begin{definition}
\cite{article.7} A statistical submanifold $N$ of a holomorphic statistical manifold $\overline{N}$ is called a holomorphic submanifold $(\mathcal{F} = 0$ and $t = 0)$ if the almost complex structure $\mathcal{J}$ of $\overline{N}$ carries each tangent space of $N$ into itself, whereas it is said to be a totally real submanifold $(\mathcal{P} = 0)$ if the almost complex structure $\mathcal{J}$ of $\overline{N}$ carries each tangent space of $N$ into its corresponding normal space. If $\mathcal{J}(TN)= T^{\perp}N$, then $N$ is called a Lagrangian submanifold ($\mathcal{P} = 0$ and $f = 0$). A CR-statistical manifold is said to be generic if $\mathcal{J}\mathcal{D}^{\perp} = T^{\perp}N$ and $\mathcal{D} \neq 0$ $(f = 0)$. If $\mathcal{D} \neq 0$ and $\mathcal{D}^{\perp} \neq 0$, then $N$ is said to be proper.
\end{definition}


For a CR-statistical submanifold $N$ we shall denote by $\mu$ the orthogonal complementary subbundle of $\mathcal{J}\mathcal{D}^{\perp}$ in $T^{\perp}N$, we have
\begin{eqnarray}
T^{\perp}N = \mathcal{J}\mathcal{D}^{\perp} \oplus \mu.
\end{eqnarray}

\begin{definition}
\cite{article.7} A holomorphic statistical manifold $(\overline{N}, \overline{\nabla}, g, \mathcal{J})$ is said to be of constant holomorphic curvature $c \in \mathbb{R}$ if the following curvature equation holds
\begin{eqnarray}\label{d}
\overline{\mathcal{S}}(X, Y)Z &=& \frac{c}{4} \big\{g(Y, Z)X - g(X, Z)Y + g(\mathcal{J}Y, Z)\mathcal{J}X \nonumber \\
&&- g(\mathcal{J}X, Z)\mathcal{J}Y + 2g(X, \mathcal{J}Y)\mathcal{J}Z \big\}.
\end{eqnarray}
for any $X, Y, Z \in \Gamma(T\overline{N})$. It is denoted by $\overline{N}(c)$.
\end{definition}

\section{Some Basic Results}

We prove the following results:

\begin{proposition}
Let $N$ be a statistical submanifold of a holomorphic statistical manifold $\overline{N}$. Then
\begin{eqnarray*}
\mathcal{A}_{V}^{*}tU = \mathcal{A}_{U}^{*}tV,
\end{eqnarray*}
for any $U, V \in \Gamma(T^{\perp}N)$ if and only if
\begin{eqnarray*}
\nabla_{X}^{\perp}fV = f\nabla_{X}^{\perp *}V,
\end{eqnarray*}
for any $X \in \Gamma(TN)$ and $V \in \Gamma(T^{\perp}N)$.
\end{proposition}

\begin{proof}
From \cite{article.7}, we have
\begin{eqnarray*}
g(\nabla_{X}^{\perp}fV, U) = g(f\nabla_{X}^{\perp *}V, U) - g(\mathcal{F}\mathcal{A}^{*}_{V}X, U) - g(\mathcal{B}(X, tV), U).
\end{eqnarray*}
Further, we derive that
\begin{eqnarray*}
g(\mathcal{A}^{*}_{V}X, tU) - g(\mathcal{A}_{U}^{*}tV, X) = 0.
\end{eqnarray*}
Our assertion follows from the last relation and self-adjoint property of $\mathcal{A}^{*}$.
\end{proof}

\begin{proposition}\label{1}
Let $N$ be a statistical submanifold of a holomorphic statistical manifold $\overline{N}$. Then
\begin{eqnarray*}
\nabla_{X}^{\perp}(\mathcal{F}Y) = \mathcal{F}\nabla_{X}^{*}Y,
\end{eqnarray*}
for any $X, Y \in \Gamma(TN)$ if and only if
\begin{eqnarray*}
\nabla_{X}(tV) = t\nabla_{X}^{\perp *}V,
\end{eqnarray*}
for any $X \in \Gamma(TN)$ and $V \in \Gamma(T^{\perp}N)$.
\end{proposition}

\begin{proof}
From \cite{article.7}, we have
\begin{eqnarray*}
g(\nabla_{X}^{\perp}(\mathcal{F}Y) - \mathcal{F}\nabla_{X}^{*}Y, V) &=& g(f\mathcal{B}^{*}(X, Y), V) - g(\mathcal{B}(X, \mathcal{P}Y), V)\\
&=& - g(\mathcal{B}^{*}(X, Y), fV) - g(\mathcal{A}^{*}_{V}X, \mathcal{P}Y)\\
&=& - g(\mathcal{A}_{fV}X, Y) + g(\mathcal{P}\mathcal{A}^{*}_{V}X, Y)\\
&=& - g(\nabla_{X}(tV) - t\nabla_{X}^{\perp *}V, Y).
\end{eqnarray*}
This completes the proof of our proposition.
\end{proof}

\begin{proposition}
Let $N$ be a holomorphic statistical submanifold of a holomorphic statistical manifold $\overline{N}$. Then
\begin{enumerate}
\item[(a)] $tr_{g}\mathcal{B} = tr_{g}\mathcal{B}^{*} = 0$,
\item[(b)] $g(\mathcal{S}(X, \mathcal{J}X)\mathcal{J}X, X) = g(\overline{\mathcal{S}}(X, \mathcal{J}X)\mathcal{J}X, X) - 2 g(\mathcal{B}(X, X), \mathcal{B}^{*}(X, X))$, for $X \in \Gamma(TN)$.
\end{enumerate}
\end{proposition}

\begin{proof}
Let $\{e_{1}, \dots, e_{n}, \mathcal{J}e_{1}, \dots, \mathcal{J}e_{n}\}$ be a local orthonormal frame on $N$. Then we have
\begin{eqnarray*}
tr_{g}\mathcal{B} &=& \sum_{i = 1}^{n}(\mathcal{B}(e_{i},e_{i}) + \mathcal{B}(\mathcal{J}e_{i}, \mathcal{J}e_{i})) = 0,
\end{eqnarray*}
and
\begin{eqnarray*}
tr_{g}\mathcal{B}^{*} &=& \sum_{i = 1}^{n}(\mathcal{B}^{*}(e_{i},e_{i}) + \mathcal{B}^{*}(\mathcal{J}e_{i}, \mathcal{J}e_{i}))\\
&=& - \sum_{i = 1}^{n}(\mathcal{J}\mathcal{B}^{*}(e_{i},\mathcal{J}e_{i}) - \mathcal{J}\mathcal{B}^{*}(\mathcal{J}e_{i}, e_{i})) = 0.
\end{eqnarray*}
Following \cite{article.7}, we have
\begin{eqnarray*}
g(\mathcal{S}(X, \mathcal{J}X)\mathcal{J}X, X) &=& g(\overline{\mathcal{S}}(X, \mathcal{J}X)\mathcal{J}X, X) - g(\mathcal{B}(\mathcal{J}X,X), \mathcal{B}^{*}(X, \mathcal{J}X))\\
&&+ \frac{1}{2}( g(\mathcal{B}^{*}(X,X), \mathcal{B}(\mathcal{J}X, \mathcal{J}X)) + g(\mathcal{B}(X,X), \mathcal{B}^{*}(\mathcal{J}X, \mathcal{J}X))) \\
&=& g(\overline{\mathcal{S}}(X, \mathcal{J}X)\mathcal{J}X, X) - g(\mathcal{B}^{*}(X,X), \mathcal{B}(X,X))\\
&&- g(\mathcal{B}(\mathcal{J}X,X), \mathcal{B}^{*}(X, \mathcal{J}X))\\
&=& g(\overline{\mathcal{S}}(X, \mathcal{J}X)\mathcal{J}X, X) - 2g(\mathcal{B}^{*}(X,X), \mathcal{B}(X,X)),
\end{eqnarray*}
for any $X \in \Gamma(TN)$.
\end{proof}

\begin{proposition}\label{p1}
On a CR-statistical submanifold $N$ of a holomorphic statistical manifold $\overline{N}$, the following
\begin{eqnarray}\label{e21}
\mathcal{A}_{fV}X = - \mathcal{A}^{*}_{V}\mathcal{P}X
\end{eqnarray}
holds for any $X \in \mathcal{D}$ and $V \in \Gamma(fT^{\perp}N)$.
\end{proposition}

\begin{proof}
From \cite{article.7}, it follows that
\begin{eqnarray*}
0 &=& g(\nabla_{Y}(tV) - t\nabla_{Y}^{\perp *}V, X)\\
&=& g(\mathcal{A}_{fV}Y - \mathcal{P}\mathcal{A}^{*}_{V}Y, X)\\
&=& g(\mathcal{A}_{fV}X, Y) + g(\mathcal{A}^{*}_{V}\mathcal{P}X, Y),
\end{eqnarray*}
for any $X \in \mathcal{D}$, $Y \in \Gamma(TN)$ and $V \in \Gamma(fT^{\perp}N)$. Thus, we get the desired result.
\end{proof}

\begin{proposition}\label{4}
Let $N$ be a CR-statistical submanifold of a holomorphic statistical manifold $\overline{N}$. Then
\begin{eqnarray*}
\nabla_{X}^{\perp}\mathcal{J}Y - \nabla_{Y}^{\perp}\mathcal{J}X \in \mathcal{J}\mathcal{D}^{\perp},
\end{eqnarray*}
for any $X, Y \in \mathcal{D}^{\perp}$.
\end{proposition}

\begin{proof}
For any $X, Y \in \mathcal{D}^{\perp}$ and $V \in \Gamma(\mu)$, we obtain
\begin{eqnarray*}
g(\nabla_{X}^{\perp}\mathcal{J}Y - \nabla_{Y}^{\perp}\mathcal{J}X, V) &=& g(\overline{\nabla}_{X}\mathcal{J}Y - \overline{\nabla}_{Y}\mathcal{J}X, V) \\
&=& g(\mathcal{J}(\overline{\nabla}_{X}^{*}Y - \overline{\nabla}_{Y}^{*}X), V)\\
&=& - g(\overline{\nabla}_{X}^{*}Y - \overline{\nabla}_{Y}^{*}X, \mathcal{J}V)\\
&=& - g(\mathcal{B}^{*}(X, Y) - \mathcal{B}^{*}(X, Y), \mathcal{J}V)\\
&=& 0.
\end{eqnarray*}
This proves our assertion.
\end{proof}

\begin{definition}
Let $S$ be a $r$-dimensional distribution on a statistical manifold $N$. Then $S$ is said to be
\begin{itemize}
\item[(a)] minimal with respect to $\nabla$ if \hspace{0.01 cm} $\sum_{i = 1}^{r} \nabla_{e_{i}}e_{i} \in S$.
\item[(a)$^*$] minimal with respect to $\nabla^{*}$ if \hspace{0.01 cm} $\sum_{i = 1}^{r} \nabla_{e_{i}}^{*}e_{i} \in S$.
\end{itemize}
\end{definition}

For minimality of the holomorphic distribution $\mathcal{D}$ with respect to $\nabla$ and $\nabla^{*}$, we have the following:

\begin{proposition}\label{12}
Let $N$ be a CR-statistical submanifold of a holomorphic statistical manifold $\overline{N}$, then the holomorphic distribution $\mathcal{D}$ is minimal with respect to $\nabla$ and $\nabla^{*}$.
\end{proposition}

\begin{proof}
For any $X \in \mathcal{D}$ and $Z \in \mathcal{D}^{\perp}$, we have
\begin{eqnarray}\label{9}
g(Z, \nabla_{X}^{*}X) &=& g(Z, \overline{\nabla}_{X}^{*}X) \nonumber \\
&=& g(\mathcal{J}Z, \mathcal{J}\overline{\nabla}_{X}^{*}X) \nonumber \\
&=& g(\mathcal{J}Z, \overline{\nabla}_{X}\mathcal{J}X) \nonumber \\
&=& - g(\overline{\nabla}_{X}^{*}\mathcal{J}Z, \mathcal{J}X) \nonumber \\
&=& g(\mathcal{A}_{\mathcal{J}Z}^{*}X, \mathcal{J}X),
\end{eqnarray}
where we have used
\begin{eqnarray*}
0 = Xg(\mathcal{J}Z, \mathcal{J}X) = g(\mathcal{J}Z, \overline{\nabla}_{X}\mathcal{J}X) + g(\overline{\nabla}_{X}^{*}\mathcal{J}Z, \mathcal{J}X).
\end{eqnarray*}
Replacing $X$ by $\mathcal{J}X$ in equation (\ref{9}), we find that
\begin{eqnarray}\label{10}
g(Z, \nabla_{\mathcal{J}X}^{*}\mathcal{J}X) = - g(\mathcal{A}_{\mathcal{J}Z}^{*}X, \mathcal{J}X).
\end{eqnarray}
By adding (\ref{9}) and (\ref{10}), we get
\begin{eqnarray}\label{1.1}
g(Z, \nabla_{X}^{*}X + \nabla_{\mathcal{J}X}^{*}\mathcal{J}X) = 0.
\end{eqnarray}
From (\ref{1.1}), we conclude that $\mathcal{D}$ is minimal with respect to $\nabla^{*}$. By simple computation, we also prove that $\mathcal{D}$ is minimal with respect to $\nabla$. This completes our proof.
\end{proof}

\begin{definition}
A distribution $S$ is called involutive if $[X, Y] \in S$ for any $X, Y \in S$.
\end{definition}

\begin{lemma}\label{14}
\cite{article.7} Let $N$ be a CR-statistical submanifold of a holomorphic statistical manifold $\overline{N}$. If $\mathcal{D}$ is involutive, then
\begin{eqnarray}\label{22}
\mathcal{B}(\mathcal{J}X, Y) = \mathcal{B}(\mathcal{J}Y, X)
\end{eqnarray}
holds for any $X, Y \in \mathcal{D}$.
\end{lemma}

In the light of Lemma \ref{14}, we have the following proposition:

\begin{proposition}\label{p}
Let $N$ be a CR-statistical submanifold of a holomorphic statistical manifold $\overline{N}$. Suppose that $\mathcal{D}$ is involutive. If $N$ is mixed totally geodesic with respect to $\overline{\nabla}$, then $\mathcal{A}_{V}^{*}\mathcal{J}X = - \mathcal{J}\mathcal{A}_{V}^{*}X$.  If $N$ is mixed totally geodesic with respect to $\overline{\nabla}^{*}$, then $\mathcal{A}_{V}\mathcal{J}X = - \mathcal{J}\mathcal{A}_{V}X$, for any $X \in \mathcal{D}$ and $V \in \Gamma(\mu)$.
\end{proposition}

\begin{proof}
For any $X \in \mathcal{D}$, $Y \in \mathcal{D}$ and $V \in \Gamma(\mu)$, we obtain
\begin{eqnarray*}
g(\mathcal{J}\mathcal{A}_{V}^{*}X, Y) &=& - g(\mathcal{A}_{V}^{*}X, \mathcal{J}Y)\\
&=& - g(\mathcal{B}(X, \mathcal{J}Y), V)\\
&=& - g(\mathcal{B}(\mathcal{J}X, Y), V)\\
&=& - g(\mathcal{A}_{V}^{*}\mathcal{J}X, Y),
\end{eqnarray*}
where we have used Lemma \ref{14}. Hence, our first assertion follows. In the same manner, we prove our second assertion.
\end{proof}

\section{Riemannian Products in Holomorphic Statistical Manifolds}

The statistical version of CR-product defined in \cite{article.50} is as follows:

\begin{definition}
Let $N$ be a CR-statistical submanifold in a holomorphic statistical manifold $\overline{N}$. Then $N$ is called a CR-product if it is a CR-product
as a CR-submanifold in a Kaehler manifold $\overline{N}$, that is, it is locally a Riemannian product of a holomorphic submanifold $N^{T}$ and a totally real submanifold $N^{\perp}$ of $\overline{N}$.
\end{definition}

We recall some results from \cite{article.7}.

\begin{proposition}\label{pp}
Let $(N, \nabla, g)$ be a CR-statistical submanifold in a holomorphic statistical manifold $\overline{N}$. If $N$ is mixed totally geodesic with respect to $\overline{\nabla}$ and $\overline{\nabla}^{*}$, then each leaf of $\mathcal{D}^{\perp}$ is totally geodesic in $N$ with respect to $\nabla$ and $\nabla^{*}$.
\end{proposition}

\begin{proposition}\label{pp1}
Let $(N, \nabla, g)$ be a CR-statistical submanifold in a holomorphic statistical manifold $\overline{N}$. If $N$ is $\mathcal{D}-$totally geodesic with respect to $\overline{\nabla}$ and $\overline{\nabla}^{*}$, then $\mathcal{D}$ is completely integrable and each leaf of $\mathcal{D}$ is totally geodesic in $N$ with respect to $\nabla$ and $\nabla^{*}$.
\end{proposition}

Now, if we assume that $\mathcal{A}_{\mathcal{F}\mathcal{D}^{\perp}} \mathcal{D} = 0$ and $\mathcal{A}_{\mathcal{F}\mathcal{D}^{\perp}}^{*} \mathcal{D} = 0$, then for any $X \in \mathcal{D}, Y \in \Gamma(TN)$ and $Z \in \mathcal{D}^{\perp}$, we have
\begin{eqnarray}\label{eq1}
g(\mathcal{B}(X, Y), \mathcal{F}Z) = g(\mathcal{B}^{*}(X, Y), \mathcal{F}Z) = 0.
\end{eqnarray}
In particular, if $Y \in \mathcal{D}$, then the relations in (\ref{eq1}) give $N$ is $\mathcal{D}-$totally geodesic with respect to $\overline{\nabla}$ and $\overline{\nabla}^{*}$. Thus, from Proposition \ref{pp1}, we know that $\mathcal{D}$ is completely integrable and each leaf of $\mathcal{D}$ is totally geodesic in $N$ with respect to $\nabla$ and $\nabla^{*}$. Now, for $Y \in \mathcal{D}^{\perp}$, again relations in (\ref{eq1}) imply that $N$ is mixed totally geodesic with respect to $\overline{\nabla}$ and $\overline{\nabla}^{*}$. From this and Proposition \ref{pp}, we can easily say that each leaf of $\mathcal{D}^{\perp}$ is totally geodesic in $N$ with respect to $\nabla$ and $\nabla^{*}$. Hence, $N$ is CR-product in $\overline{N}$.

We can state the following:

\begin{theorem}\label{11}
A CR-statistical submanifold $N$ of a holomorphic statistical manifold $\overline{N}$ is a CR-product if
\begin{eqnarray*}
\mathcal{A}_{\mathcal{F}\mathcal{D}^{\perp}} \mathcal{D} = \mathcal{A}_{\mathcal{F}\mathcal{D}^{\perp}}^{*} \mathcal{D} = 0.
\end{eqnarray*}
\end{theorem}

Next, we prove the following:

\begin{theorem}
Let $N$ be an $n$-dimensional generic statistical submanifold of a $2m-$dimensional holomorphic statistical manifold $\overline{N}$. If
\begin{eqnarray}\label{y}
\mathcal{P}\nabla_{X}^{*}Y = \nabla_{X}\mathcal{P}Y
\end{eqnarray}
holds, and $\mathcal{B}^{*}(X, \mathcal{D}) = 0$, for any $X, Y \in \Gamma(TN)$. Then $N$ is locally a Riemannian product of a holomorphic submanifold $N^{T}$ and a totally real submanifold $N^{\perp}$ of $\overline{N}$.
\end{theorem}

\begin{proof}
Following (\ref{y}) and \cite{article.7}, we arrive at
\begin{eqnarray*}
\mathcal{A}_{\mathcal{F}Y}X = - t \mathcal{B}^{*}(X, Y) = - \mathcal{J}\mathcal{B}^{*}(X, Y),
\end{eqnarray*}
for any $X, Y \in \Gamma(TN)$. Replacing $Y$ by $\mathcal{P}Y$, we get
\begin{eqnarray}\label{ya}
\mathcal{B}^{*}(X, \mathcal{P}Y) = 0.
\end{eqnarray}
For any $Y \in \mathcal{D}$ and $X \in \Gamma(TN)$, we obtain
\begin{eqnarray*}\label{e4}
\mathcal{F}\nabla_{X}Y = \mathcal{B}^{*}(X, \mathcal{P}Y) - f\mathcal{B}(X, Y) = 0,
\end{eqnarray*}
and
\begin{eqnarray*}\label{e4}
\mathcal{F}\nabla_{X}^{*}Y = \mathcal{B}(X, \mathcal{P}Y) - f\mathcal{B}^{*}(X, Y) = 0.
\end{eqnarray*}
Thus, the distribution $\mathcal{D}$ is parallel with respect to $\nabla$. Now, we consider $Y \in \mathcal{D}^{\perp}$, then for any $X, Z \in \Gamma(TN)$, we have
\begin{eqnarray*}\label{e5}
g(\mathcal{P}\nabla_{X}^{*}Y, \mathcal{P}Z) &=& - g(\mathcal{B}^{*}(X, \mathcal{P}Z), \mathcal{F}Y) + g(\mathcal{B}^{*}(X, Y), \mathcal{FP}Z) = 0,
\end{eqnarray*}
and
\begin{eqnarray*}\label{e5}
g(\mathcal{P}\nabla_{X}Y, \mathcal{P}Z) &=& - g(\mathcal{B}(X, \mathcal{P}Z), \mathcal{F}Y) + g(\mathcal{B}(X, Y), \mathcal{FP}Z) = 0,
\end{eqnarray*}
Thus, the distribution $\mathcal{D}^{\perp}$ is parallel with respect to $\nabla^{*}$. Consequently, $N$ is locally a Riemannian product of a holomorphic submanifold $N^{T}$ and a totally real submanifold $N^{\perp}$ of $\overline{N}$.
\end{proof}

\section{Chen-Ricci Inequality for Statistical Submanifolds}

In this section, we obtain an optimal inequality for Ricci curvature in terms of squared norm of mean curvature and canonical structure by using an optimization technique. Following \cite{oprea}, we have

\begin{theorem}\label{oth}
Let $(N, g)$ be  a Riemannian submanifold of a Riemannian manifold $(\overline{N}, \overline{g})$ and $f : N \rightarrow \mathbb{R}$ be a differentiable function. If $x_{0} \in N$ is the solution of the optimum problem $\min\limits_{x_{0} \in N} f(x_{0})$, then
\begin{enumerate}
\item[(a)] $(grad \hspace{0.1 cm} f)(x) \in T_{x}^{\perp}N$,
\item[(b)] the bilinear form $A : T_{x}N \times T_{x}N \rightarrow \mathbb{R}$,
$$A(X, Y) = Hess_{f}(X, Y) + \overline{g}(h^{'}(X, Y), (grad \hspace{0.1 cm} f)(x))$$
\end{enumerate}
is positive semi-definite, where $h^{'}$ is the second fundamental form of $N$ in $\overline{N}$ and $grad \hspace{0.1 cm} f$ denotes the gradient of $f$.
\end{theorem}

For nice applications of the above result, follow \cite{toprea}. Here we prove the following optimal inequality in the light of Theorem \ref{oth}:

\begin{theorem}\label{t}
Let $(N, \nabla, g)$ be an $m$-dimensional statistical submanifold in a holomorphic statistical manifold $\overline{N}(c)$ of constant holomorphic sectional curvature $c$. For each unit vector $X \in T_{x}N$, $x \in N$, we have
\begin{eqnarray}\label{r}
Ric^{\nabla, \nabla^{*}}(X) &\geq& 2 Ric^{0}(X) - \frac{c}{4}(m - 1 + 3||\mathcal{P} X||^{2}) \nonumber  \\
&&- \frac{m^{2}}{8} [||\mathcal{H}||^{2} + ||\mathcal{H}^{*}||^{2}],
\end{eqnarray}
where $Ric^{0}$ denotes the Ricci curvature with respect to Levi-Civita connection.
\end{theorem}

\begin{proof}
We choose $\{e_{1}, \dots, e_{m}\}$ as the orthonormal frame of $T_{x}N$ such that $e_{1} = X$ and $||X|| = 1$, and $\{e_{m+1}, \dots, e_{2n}\}$ as the the orthonormal frame of $T_{x}N$ in $\overline{N}$. Then by (\ref{g}) and (\ref{g1}), we have
\begin{eqnarray*}
2\overline{\mathcal{S}}(e_{1}, e_{i}, e_{1}, e_{i}) &=& 2\mathcal{S}(e_{1}, e_{i}, e_{1}, e_{i}) - g(\mathcal{B}(e_{1}, e_{1}), \mathcal{B}^{*}(e_{i}, e_{i}))\\
&& - g(\mathcal{B}^{*}(e_{1}, e_{1}), \mathcal{B}(e_{i}, e_{i}))
+ 2g(\mathcal{B}(e_{1}, e_{i}), \mathcal{B}^{*}(e_{1}, e_{i}))\\
&=& 2\mathcal{S}(e_{1}, e_{i}, e_{1}, e_{i}) - \{4g(\mathcal{B}^{0}(e_{1}, e_{1}), \mathcal{B}^{0}(e_{i}, e_{i}))\\
&&- g(\mathcal{B}(e_{1}, e_{1}), \mathcal{B}(e_{i}, e_{i})) - g(\mathcal{B}^{*}(e_{1}, e_{1}), \mathcal{B}^{*}(e_{i}, e_{i}))\\
&&- 4g(\mathcal{B}^{0}(e_{1}, e_{i}), \mathcal{B}^{0}(e_{1}, e_{i})) + g(\mathcal{B}(e_{1}, e_{i}), \mathcal{B}(e_{1}, e_{i}))\\
&&+ g(\mathcal{B}^{*}(e_{1}, e_{i}), \mathcal{B}^{*}(e_{1}, e_{i}))\}\\
&=& 2\mathcal{S}(e_{1}, e_{i}, e_{1}, e_{i}) - 4\sum_{r=m+1}^{2n}(h^{0r}_{11}h^{0r}_{ii} - (h^{0r}_{1i})^{2})\\
&&+ \sum_{r=m+1}^{2n}(h^{r}_{11}h^{r}_{ii} - (h^{r}_{1i})^{2}) + \sum_{r=m+1}^{2n}(h^{*r}_{11}h^{*r}_{ii} - (h^{*r}_{1i})^{2}),
\end{eqnarray*}
where we have used the notations $\overline{\mathcal{S}}(X,Y,Z,W) = g(\overline{\mathcal{S}}(X,Y)W,Z)$. Summing over $2 \leq i \leq m$, we have
\begin{eqnarray*}
\frac{c}{2}(m-1 + 3||\mathcal{P} X||^{2}) &=& 2Ric^{\nabla, \nabla^{*}}(X) - 4\sum_{r=m+1}^{2n}\sum_{i=2}^{m}(h^{0r}_{11}h^{0r}_{ii} - (h^{0r}_{1i})^{2}) \\
&&+ \sum_{r=m+1}^{2n}\sum_{i=2}^{m}(h^{r}_{11}h^{r}_{ii} - (h^{r}_{1i})^{2}) + \sum_{r=m+1}^{2n}\sum_{i=2}^{m}(h^{*r}_{11}h^{*r}_{ii} \\
&&- (h^{*r}_{1i})^{2}),
\end{eqnarray*}
where $Ric^{\nabla, \nabla^{*}}(X)$ denotes the Ricci curvature of $N$ with respect to $\nabla$ and $\nabla^{*}$ at $x$. Further, we derive
\begin{eqnarray}\label{g2}
2Ric^{\nabla, \nabla^{*}}(X) - \frac{c}{2}(m-1 + 3||\mathcal{P} X||^{2}) &=& 4\sum_{r=m+1}^{2n}\sum_{i=2}^{m}(h^{0r}_{11}h^{0r}_{ii} - (h^{0r}_{1i})^{2}) \nonumber \\
&&- \sum_{r=m+1}^{2n}\sum_{i=2}^{m}(h^{r}_{11}h^{r}_{ii} - (h^{r}_{1i})^{2})\nonumber \\
&&- \sum_{r=m+1}^{2n}\sum_{i=2}^{m}(h^{*r}_{11}h^{*r}_{ii} - (h^{*r}_{1i})^{2}).\nonumber \\
\end{eqnarray}
By Gauss equation with respect to Levi-Civita connection, it follows that
\begin{eqnarray*}
Ric^{0}(X) - \frac{c}{4}(m-1 + 3||\mathcal{P} X||^{2}) = \sum_{r=m+1}^{n}\sum_{i=2}^{m}(h^{0r}_{11}h^{0r}_{ii} - (h^{0r}_{1i})^{2}).
\end{eqnarray*}
Substituting into (\ref{g2}), we arrive at
\begin{eqnarray*}
2Ric^{\nabla, \nabla^{*}}(X) - \frac{c}{2}(m-1 + 3||\mathcal{P} X||^{2}) &=& 4[Ric^{0}(X) - \frac{c}{4}(m-1 + 3||\mathcal{P} X||^{2})] \\
&&- \sum_{r=m+1}^{2n}\sum_{i=2}^{m}(h^{r}_{11}h^{r}_{ii} - (h^{r}_{1i})^{2})\nonumber \\
&&- \sum_{r=m+1}^{2n}\sum_{i=2}^{m}(h^{*r}_{11}h^{*r}_{ii} - (h^{*r}_{1i})^{2}).\nonumber \\
\end{eqnarray*}
On simplifying the previous relation, we get
\begin{eqnarray}\label{g3}
&&-2Ric^{\nabla, \nabla^{*}}(X) - \frac{c}{2}(m-1 + 3||\mathcal{P} X||^{2}) + 4Ric^{0}(X)\nonumber \\
&=& \sum_{r=m+1}^{n}\sum_{i=2}^{m}(h^{r}_{11}h^{r}_{ii} - (h^{r}_{1i})^{2})\nonumber \\
&&+ \sum_{r=m+1}^{n}\sum_{i=2}^{m}(h^{*r}_{11}h^{*r}_{ii} - (h^{*r}_{1i})^{2})\nonumber \\
&\leq& \sum_{r=m+1}^{n}\sum_{i=2}^{m}h^{r}_{11}h^{r}_{ii} + \sum_{r=m+1}^{n}\sum_{i=2}^{m}h^{*r}_{11}h^{*r}_{ii}.
\end{eqnarray}
Let us define the quadratic form $f_{r}, f_{r}^{*} : \mathbb{R}^{m} \rightarrow \mathbb{R}$ by
$$f_{r}(h_{11}^{r}, h_{22}^{r}, \dots, h_{mm}^{r}) = \sum_{r=m+1}^{2n}\sum_{i=2}^{m}h^{r}_{11}h^{r}_{ii},$$ and
$$f_{r}^{*}(h_{11}^{*r}, h_{22}^{*r}, \dots, h_{mm}^{*r}) = \sum_{r=m+1}^{2n}\sum_{i=2}^{m}h^{*r}_{11}h^{*r}_{ii}.$$
We consider the constrained extremum problem $\max f_{r}$ subject to $$Q : \sum_{i=1}^{m}h^{r}_{ii} = \alpha^{r},$$ where $\alpha^{r}$ is a real constant. The gradient vector field of the function $f_{r}$ is given by $$grad \hspace{0.1 cm} f_{r} = (\sum_{i=2}^{m}h^{r}_{ii}, h^{r}_{11}, h^{r}_{11}, \dots, h^{r}_{11}).$$ For an optimal solution $p = (h_{11}^{r}, h_{22}^{r}, \dots h_{mm}^{r})$ of the problem in question, the vector $grad \hspace{0.1 cm} f_{r}$ is normal to $Q$ at the point $p$. It follows that $$h^{r}_{11} = \sum_{i=2}^{m}h^{r}_{ii} = \frac{\alpha^{r}}{2}.$$
Now, we fix $x \in Q$. The bilinear form $A : T_{x}Q \times T_{x}Q \rightarrow \mathbb{R}$ has the following expression:
\begin{eqnarray*}
A(X, Y) = Hess_{f_{r}}(X, Y) + <h^{'}(X, Y), (grad \hspace{0.1 cm} f_{r})(x)>,
\end{eqnarray*}
where $h^{'}$ denotes the second fundamental form of $Q$ in $\mathbb{R}^{m}$ and $<\cdot,\cdot>$ denotes the standard inner product on $\mathbb{R}^{m}$. The Hessian matrix of $f_{r}$ is given by
\begin{eqnarray*}
Hess_{f_{r}} = \left(
 \begin{array}{ccccc}
    0 & 1 & \dots & 1 \\
    1 & 0 & \dots& 0 \\
    \vdots & \vdots \ & \ddots & \vdots \\
    1 & 0  & \dots &  0 \\
    1 & 0  & \dots &  0
  \end{array}
\right).
\end{eqnarray*}
We consider a vector $X \in T_{x}Q$, which satisfies a relation $\sum_{i=2}^{m} X_{i} = -X_{1}$. As $h^{'} = 0$ in $\mathbb{R}^{m}$, we get
\begin{eqnarray*}
A(X, X) = Hess_{f_{r}}(X, X) &=& 2 \sum_{i=2}^{m} X_{1}X_{i} \\
&=& (X_{1} + \sum_{i=2}^{m} X_{i})^{2} - (X_{1})^{2} - (\sum_{i=2}^{m} X_{i})^{2}\\
&=& - 2(X_{1})^{2} \leq 0.
\end{eqnarray*}
However, the point $p$ is the only optimal solution, i.e., the global maximum point of problem. Thus, we obtain
\begin{eqnarray}\label{g4}
f_{r} \leq \frac{1}{4}(\sum_{i=1}^{m} h_{ii}^{r})^{2} = \frac{m^{2}}{4}(H^{r})^{2},
\end{eqnarray}
Next, we deal with the constrained extremum problem $\max f_{r}^{*}$ subject to $$Q^{*} : \sum_{i=1}^{m}h^{*r}_{ii} = \alpha^{*r},$$ where $\alpha^{*r}$ is a real constant. By similar arguments as above, we find
\begin{eqnarray}\label{g410}
f_{r}^{*} \leq \frac{1}{4}(\sum_{i=1}^{m} h_{ii}^{*r})^{2} = \frac{m^{2}}{4}(\mathcal{H}^{*r})^{2}.
\end{eqnarray}
On combining (\ref{g3}), (\ref{g4}) and (\ref{g410}), we get our desired inequality (\ref{r}).
\end{proof}

The characterisation of equality cases in Theorem \ref{t}:
\begin{theorem}
Let $(N, \nabla, g)$ be an $m$-dimensional statistical submanifold in a holomorphic statistical manifold $\overline{N}(c)$ of constant holomorphic sectional curvature $c$. The equality holds in the inequality (\ref{r}) if and only if $$h(X, X) = \frac{m}{2} \mathcal{H}(x), \textrm{             } h^{*}(X, X) = \frac{m}{2} \mathcal{H}^{*}(x)$$ and $$h(X, Y) = 0, \textrm{               } h^{*}(X, Y) = 0$$ for all $Y \in T_{x}N$ orthogonal to $X$.
\end{theorem}

\begin{proof}
The vector field $X$ satisfies the equality case if and only if
$$h_{1i}^{r} = 0, \textrm{             } h_{1i}^{*r} = 0, \textrm{             } i \in \{2, \dots, m\},$$
and
$$h_{11}^{r} = \sum_{i=2}^{m}h_{ii}^{r}, \textrm{             } h_{11}^{*r} = \sum_{i=2}^{m}h_{ii}^{*r}, \textrm{             } r \in \{m+1, \dots, 2n\},$$
which can be rewritten as
$$h_{11}^{r} = \frac{m}{2}\mathcal{H},$$
and $$h_{11}^{*r} = \frac{m}{2}\mathcal{H}^{*}.$$
Thus, it proves our assertion.
\end{proof}

An immediate consequence of Theorem \ref{t} is as follows:

\begin{corollary}
Let $(N, \nabla, g)$ be an $m$-dimensional CR-statistical submanifold in a holomorphic statistical manifold $\overline{N}(c)$ of constant holomorphic sectional curvature $c$. Then
\begin{enumerate}
\item[(a)] For each unit vector $X \in \mathcal{D}_{x}$
\begin{eqnarray*}
Ric^{\nabla, \nabla^{*}}(X) \geq 2 Ric^{0}(X) - \frac{c}{4}(m +2) - \frac{m^{2}}{8} [||\mathcal{H}||^{2} + ||\mathcal{H}^{*}||^{2}].
\end{eqnarray*}
\item[(b)] for each unit vector $X \in \mathcal{D}_{x}^{\perp}$
\begin{eqnarray*}
Ric^{\nabla, \nabla^{*}}(X) \geq 2 Ric^{0}(X) -  \frac{c}{4}(m -1) - \frac{m^{2}}{8} [||\mathcal{H}||^{2} + ||\mathcal{H}^{*}||^{2}].
\end{eqnarray*}
\end{enumerate}
\end{corollary}

\section{Some Related Examples}

\begin{example}
Let $(\overline{N}, g)$ be a family of exponential distributions of mean 0:
$$\overline{N}:= \{p(u,\Phi)  |   p(u,\Phi)=\Phi e^{-\Phi u}, u \in [0,\infty), \Phi \in (0, \infty)\},$$
a Riemannian metric is given by $g:= \Phi^{-2}(d \Phi)^{2}$, and an $\alpha-$connection $(\alpha \in \mathbb{R})$ on $\overline{N}$ is defined by
$$\overline{\nabla}_{\frac{\partial}{\partial \Phi}}^{\alpha}\frac{\partial}{\partial \Phi} = (\alpha -1)\Phi^{-1}\frac{\partial}{\partial \Phi}.$$
Then $(\overline{N}, \overline{\nabla}^{\alpha}, g)$ is a $1-$dimensional statistical manifold.

We remark that one can also construct examples for higher dimension by defining Fisher information metric and $\alpha-$connection on a family of statistical distribution (for example \cite{f3,article.7}).
\end{example}

We state the following lemma:

\begin{lemma}\label{L}
\cite{f3} Let $(\overline{N}, g, \mathcal{J})$ be a Kaehler manifold and a connection $\overline{\nabla}$ is defined as $\overline{\nabla}:= \nabla^{g} + K$, where $K$ is a $(1,2)$-tensor field satisfying the following conditions:
\begin{eqnarray}\label{c1}
K(X, Y) &=& K(Y, X),
\end{eqnarray}
\begin{eqnarray}\label{c2}
g(K(X, Y), Z) &=& g(Y, K(X, Z))
\end{eqnarray}
and
\begin{eqnarray}\label{c3}
K(X, \mathcal{J}Y) + \mathcal{J} K(X, Y) &=& 0,
\end{eqnarray}
for any $X, Y, Z \in \Gamma(T\overline{N})$. Then $(\overline{N}, \overline{\nabla}, g, \mathcal{J})$ is a holomorphic statistical manifold.
\end{lemma}

By using the above Lemma \ref{L}, we construct the following examples:

\begin{example}\label{1e}
\cite{article} Let us consider a Kaehler manifold $(\overline{N} = \{(x^{1}, x^{2})^{\prime} \in \mathbb{R}^{2} | x^{1} > 0\}, g, \mathcal{J})$, where a Riemanian metric $g$ and the standard complex structure $\mathcal{J}$ on $\overline{N}$ are defined by
$$g = x^{1}\{(dx^{1})^{2} + (dx^{2})^{2}\}$$
and
$$\mathcal{J} \partial_{1} = \partial_{2}, \hspace{0.2 cm} \mathcal{J} \partial_{2} = - \partial_{1},$$
where $\partial_{i} = \frac{\partial}{\partial x^{i}}$ for $i = 1, 2$. Now, for any $\lambda \in \mathbb{R}$, we define a $(1,2)$-tensor field $K$ on $\mathbb{R}^{2}$ as follows:
$$K = \sum_{i,j,l=1}^{2} k_{ij}^{l}\partial_{l}\otimes dx^{i} \otimes dx^{j},$$
where
$- k_{11}^{1} = k_{12}^{2} = k_{21}^{2} = k_{22}^{1} = \lambda$ and $k_{11}^{2} = k_{12}^{1} = k_{21}^{1} = k_{22}^{2} = 0$. Then $K$ satisfies all three conditions of Lemma \ref{L}, and hence we get a holomorphic statistical manifold $(\overline{N}, \overline{\nabla}:= \nabla^{g} + K, g, \mathcal{J})$, where an affine connection $\overline{\nabla}$ on $\overline{N}$ is given by
\begin{eqnarray*}
\overline{\nabla}_{\partial_{1}}\partial_{1} &=& \bigg(\frac{1}{2}(x^{1})^{-1} - \lambda \bigg)\partial_{1},\\
\overline{\nabla}_{\partial_{1}}\partial_{2} &=& \overline{\nabla}_{\partial_{2}}\partial_{1} = \bigg(\frac{1}{2}(x^{1})^{-1} + \lambda\bigg)\partial_{2},\\
\overline{\nabla}_{\partial_{2}}\partial_{2} &=& -\bigg(\frac{1}{2}(x^{1})^{-1} + \lambda \bigg)\partial_{1}.
\end{eqnarray*}
\end{example}

\begin{example}\label{ex}
\cite{article} Let $(g, \mathcal{J})$ be a Kaehler structure on $\overline{N}$. We take a vector field $\Lambda \in \Gamma(T\overline{N})$ and set a tensor field $K_{1} \in \Gamma(T\overline{N}^{(1,2)})$ as follows:
\begin{eqnarray}\label{e7}
K_{1}(X, Y) &=& \big[g(\mathcal{J}\Lambda, X)g(\mathcal{J}\Lambda, Y) - g(\Lambda, X)g(\Lambda, Y)\big] \Lambda \nonumber \\
&&+ \big[g(\mathcal{J}\Lambda, X)g(\Lambda, Y) + g(\Lambda, X)g(\mathcal{J}\Lambda, Y)\big]\mathcal{J}\Lambda
\end{eqnarray}
for any $X, Y \in \Gamma(T\overline{N})$. Then, by simple computation, we see that $K_{1}$ satisfies three conditions of Lemma \ref{L}, and hence a holomorphic statistical manifold $(\overline{N}, \overline{\nabla}:= \nabla^{g} + K_{1}, g, \mathcal{J})$ is obtained.
\end{example}

\begin{example}\label{ex1}
\cite{article} For a Kaehler manifold $(\overline{N}, g, \mathcal{J})$, we take a vector field $\Lambda \in \Gamma(T\overline{N})$ and set $K_{2}$ as follows:
\begin{eqnarray}\label{e8}
K_{2}(X, Y) &=& \big[g(\Lambda, \mathcal{J}X)g(\Lambda, \mathcal{J}Y) - g(\Lambda, X)g(\Lambda, Y) \nonumber \\
&&- g(\Lambda, \mathcal{J}X)g(\Lambda, Y) - g(\Lambda, X)g(\Lambda, \mathcal{J}Y)\big] \Lambda \nonumber \\
&&+ \big[g(\Lambda, X)g(\Lambda, Y) - g(\Lambda, \mathcal{J}X)g(\Lambda, \mathcal{J}Y) \nonumber \\
&&- g(\Lambda, \mathcal{J}X)g(\Lambda, Y) - g(\Lambda, X)g(\Lambda, \mathcal{J}Y)\big]\mathcal{J}\Lambda
\end{eqnarray}
for any $X, Y \in \Gamma(T\overline{N})$. Then $K_{2} \in \Gamma(T\overline{N}^{(1,2)})$ satisfies three conditions of Lemma \ref{L} as in Example \ref{ex}, and hence $(\overline{N}, \overline{\nabla}:= \nabla^{g} + K_{2}, g, \mathcal{J})$ becomes a holomorphic statistical manifold.
\end{example}

\begin{remark}
We remark that $K_{1} = K_{3}(\mathcal{J} \Lambda)$ ($K_{3} = K_{3}(\Lambda)$ is found in \cite{article.f}, as stated in Example \ref{ex2}) and $K_{2} = (1 - \mathcal{J})K_{1}$.
\end{remark}

Some examples of CR-statistical submanifolds in holomorphic statistical manifolds are given by:

\begin{example}\label{ex2}
Recall the examples of \cite{article.f}. For a Kaehler manifold $(\overline{N}, \overline{g}, \mathcal{J})$ and $\Lambda \in \Gamma(T\overline{N})$, set
\begin{eqnarray}\label{e1}
K_{3}(X, Y) &=& \big[\overline{g}(\Lambda, \mathcal{J}X)\overline{g}(\Lambda, \mathcal{J}Y) - \overline{g}(\Lambda, X)\overline{g}(\Lambda, Y)\big] \mathcal{J}\Lambda \nonumber \\
&&+ \big[\overline{g}(\Lambda, \mathcal{J}X)\overline{g}(\Lambda, Y) + \overline{g}(\Lambda, X)\overline{g}(\Lambda, \mathcal{J}Y)\big]\Lambda
\end{eqnarray}
for any $X, Y \in \Gamma(T\overline{N})$.

On the other hand, let $\mathbb{C}^{n + 1}$ be the complex Euclidean space with coordinates $z^{i} = x^{i} + \sqrt{-1}y^{i}$. For a Kaehler manifold $\mathbb{C}^{n + 1}$, we can construct $K_{3}$ as (\ref{e1}) and can get a holomorphic statistical manifold $(\mathbb{C}^{n + 1}, \overline{\nabla}:= \nabla^{\overline{g}} + K_{3}, \overline{g}, \mathcal{J})$, where $\nabla^{\overline{g}}$ is the standard statistical flat connection on $\mathbb{C}^{n + 1}$ and $\Lambda$ is the radial vector field, i.e., $\Lambda = \sum x^{i}\frac{\partial}{\partial x^{i}} + \sum y^{i} \frac{\partial}{\partial y^{i}}$. Now, one can define a statistical immersion as follows:
$$f: \mathbb{C}^{n} \times \mathbb{R} \mapsto \mathbb{C}^{n + 1}$$
and take a standard CR-submanifold of the form $N = \mathbb{C}^{n} \times \mathbb{R}$ in $\mathbb{C}^{n+1}$. Also one can find the induced statistical structure $(\nabla, g)$ on $N$ from $(\overline{\nabla}, \overline{g})$ by an immersion $f$. Then $(N, \nabla, g)$ becomes a CR-product in a holomorphic statistical submanifold $\mathbb{C}^{n+1}$.
\end{example}

\begin{example}\label{ex3}
Let $\mathbb{R}^{2n}$ be a holomorphic statistical manifold, which is topologically the same as $(\mathbb{R}^{+})^{2n}$. Let $\mathbb{R}^{2n}$ be considered as $\mathbb{C}^{n}$ and then we define a statistical submanifold by
{\small \begin{eqnarray*}
\big(-\frac{\pi}{2}, \frac{\pi}{2}\big)^{n} \ni {(x^{1}, \dots, x^{n})}\mapsto {(r_{1}\cos x^{1}, \dots, r_{n}\cos x^{n}, r_{1}\sin x^{1}, \dots, r_{n}\sin x^{n})}\in \mathbb{C}^{n},
\end{eqnarray*}}
where $r_{1}, \dots, r_{n}$ are positive constants. Then $\big(-\frac{\pi}{2}, \frac{\pi}{2}\big)^{n}$ is a Lagrangian submanifold of $\mathbb{C}^{n}$ (see Example $6$ in \cite{article.7}).

Furthermore, we consider a holomorphic statistical manifold $(\mathbb{C}^{m+1}, \overline{\nabla}:= \nabla^{\overline{g}} + K_{4}, \overline{g}, \mathcal{J})$, where $K_{4}$ is as follows (see \cite{article.f}):
\begin{eqnarray}\label{e2}
K_{4}(X, Y) &=& \big[\overline{g}(\Lambda, \mathcal{J}X)\overline{g}(\Lambda, \mathcal{J}Y) - \overline{g}(\Lambda, X)g(\Lambda, Y)\nonumber \\
&&+ \overline{g}(\Lambda, X)\overline{g}(\Lambda, \mathcal{J}Y) + \overline{g}(\Lambda, \mathcal{J}X)\overline{g}(\Lambda, Y)\big] \mathcal{J}\Lambda\nonumber \\
&&+ \big[\overline{g}(\Lambda, X)\overline{g}(\Lambda, Y) - \overline{g}(\Lambda, \mathcal{J}X)\overline{g}(\Lambda, \mathcal{J}Y)\nonumber \\
&&+ \overline{g}(\Lambda, X)\overline{g}(\Lambda, \mathcal{J}Y) + \overline{g}(\Lambda, \mathcal{J}X)\overline{g}(\Lambda, Y)\big]\Lambda,
\end{eqnarray}
for $X, Y \in \Gamma(T\mathbb{C}^{m+1})$. Now, one can consider a statistical immersion $i$, given by
\begin{eqnarray*}
i : \mathbb{C}^{n+1} \times \big(-\frac{\pi}{2}, \frac{\pi}{2}\big)^{n} \mapsto \mathbb{C}^{2n+1},
\end{eqnarray*}
where $\mathbb{C}^{n+1} \times \big(-\frac{\pi}{2}, \frac{\pi}{2}\big)^{n}$ is a product Riemannian manifold, denoted by $N$. Also one can show that $(N = \mathbb{C}^{n+1} \times \big(-\frac{\pi}{2}, \frac{\pi}{2}\big)^{n}, \nabla, g)$ is generic submanifold of $\mathbb{C}^{2n+1}$, where $(\nabla, g)$ is the statistical structure on $N$ induced by $i$ from $(\overline{\nabla}, \overline{g})$. Moreover, $N$ is a CR-statistical submanifold of $\mathbb{C}^{m+1}$, $m > 2n$. Thus, $N$ induces a natural CR-product in a holomorphic statistical manifold $\mathbb{C}^{m+1}$.
\end{example}

\end{document}